\documentclass[11pt,a4paper]{amsart}
\usepackage{mathrsfs}
\usepackage{syntonly}
\usepackage{amsmath}
\usepackage{amsthm}
\usepackage{amsfonts}
\usepackage{amssymb}
\usepackage{latexsym}
\usepackage{amscd,amssymb,amsopn,amsmath,amsthm,graphics,amsfonts,mathrsfs,accents,enumerate,verbatim,calc}
\usepackage[dvips]{graphicx}
\usepackage[colorlinks=true,linkcolor=red,citecolor=blue]{hyperref}
\usepackage[all]{xy}
\usepackage{enumitem}
\usepackage{url}
\date{}
\hypersetup
{colorlinks=true,linkcolor=black}
\allowdisplaybreaks[4] \footskip=15pt
\renewcommand{\uppercasenonmath}[1]{}

\numberwithin{equation}{section} \theoremstyle{plain}
\newtheorem{lem}{Lemma}[section]
\newtheorem{cor}[lem]{Corollary}
\newtheorem{prop}[lem]{Proposition}
\newtheorem{thm}[lem]{Theorem}

\newtheorem{Ex}[lem]{Example}
\newtheorem{Quest}[lem]{Question}
\newtheorem{Property}[lem]{Property}
\newtheorem{Properties}[lem]{Properties}
\newtheorem{Subprops}{}[lem]
\newtheorem{Para}[lem]{}

\newtheorem{rem}[lem]{Remark}

\newtheorem*{ack*}{ACKNOWLEDGEMENTS}

\newcommand{\pf}{\noindent\begin {proof}}
\newcommand{\epf}{\end{proof}}

\newcommand{\ra}{\rightarrow}
\newcommand{\Ext}{\mbox{\rm Ext}}
\newcommand{\Hom}{\mbox{\rm Hom}}
\newcommand{\Tor}{\mbox{\rm Tor}}
\newcommand{\im}{\mbox{\rm im}}

\newcommand{\Gpd}{\mbox{\rm Gpd}}

\newcommand{\pd}{\mbox{\rm pd}}

\newcommand{\Add}{\mbox{\rm Add}}

\begin{document}
\begin{center}
{\Large  \bf  Auslander conditions and tilting-like cotorsion pairs}

\vspace{0.5cm} Jian Wang$^{a}$, Yunxia Li$^{a}$, Jinyong Wu$^{b}$ and Jiangsheng Hu$^{c,d}$\footnote{Corresponding author}  \\
$^{a}$College of  Science, Jinling Institute of
Technology, Nanjing 211169, China\\

$^{b}$Department of Mathematics, Wenzhou University, Wenzhou 325035, China\\

$^{c}$School of Mathematics, Hangzhou Normal University, Hangzhou 311121, China\\

$^{d}$School of Mathematics and Physics, Jiangsu University of Technology,\\
 Changzhou 213001, China\\
E-mail: wangjian@jit.edu.cn, liyunxia@jit.edu.cn, jywu@wzu.edu.cn and jiangshenghu@hotmail.com
\end{center}

\bigskip
\centerline { \bf  Abstract}
\leftskip10truemm \rightskip10truemm \noindent
We study homological behavior of modules satisfying the
Auslander condition. Assume that $\mathcal{AC}$ is the class of left $R$-modules satisfying the Auslander condition. It is proved that each cycle of an exact complex with each term in $\mathcal{AC}$ belongs to $\mathcal{AC}$ for any ring $R$. As a consequence, we show that for any left Noetherian ring $R$, $\mathcal{AC}$ is a resolving subcategory of the category of left $R$-modules if and only if $_RR$ satisfies the Auslander condition if and only if each Gorenstein  projective left $R$-module belongs to $\mathcal{AC}$. As an application, we prove that, for an  Artinian algebra $R$ satisfying the Auslander condition, $R$ is Gorenstein if and only if $\mathcal{AC}$ coincides with the class of Gorenstein projective left $R$-modules if and only if (${\mathcal{AC}^{<  \infty}},(\mathcal{AC}^{<  \infty})^\bot$) is a tilting-like cotorsion pair if and only if {\rm(}${\mathcal{AC}^{< \infty}},\mathcal{I}${\rm)} is a tilting-like cotorsion pair, where $\mathcal{AC}^{<\infty}$ is the class of left $R$-modules with finite $\mathcal{AC}$-dimension and $\mathcal{I}$ is the class of injective left $R$-modules. This leads to some criteria for the validity of the Auslander and Reiten conjecture which says that an  Artinian algebra satisfying the Auslander condition is Gorenstein.
\\[2mm]

\leftskip0truemm \rightskip0truemm

\section {Introduction}
For a commutative Noetherian ring $R$, Bass proved in \cite{Bass1963} that, it is a Gorenstein ring (that is, the self-injective dimension of $R$ is finite) if and only if the flat dimension of the $i$-th term in a minimal injective coresolution of $R$ as an $R$-module is at most $i-1$ for any $i\geqslant1$. Auslander showed that this condition is left-right symmetric for a left and right Noetherian ring $R$; see Fossum, Griffith and Reiten \cite{FGR}. In this case, $R$ is said to satisfy the \emph{Auslander condition}, which may be regarded as a non-commutative version of commutative Gorenstein rings. Moreover, Auslander and Reiten conjectured in \cite{ARk-G} that an  Artinian algebra satisfying the Auslander condition is Gorenstein. The conjecture is of particular interest due to its strong connection with several long-standing homological conjectures such as the finitistic dimension conjecture, the Nakayama conjecture and the Wakamatsu tilting conjecture (see \cite{ARk-G,Beli and Reiten,Huang Auscondition}).

Cotorsion pairs were invented by Salce \cite{SL} in the category of
abelian groups and have been deeply studied in approximation theory of modules \cite{GT}, especially in the proof of the flat cover conjecture \cite{BEE}. Motivated by the fact that finitistic dimensions of an algebra can alternatively be computed by Gorenstein projective dimension, Moradifar and $\check{\rm{S}}$aroch \cite{MoraSaFin} introduced the notion of tilting-like cotorsion pairs, which is a natural generalization of tilting cotorsion pairs. These cotorsion pairs have been studied in \cite{MoraSaFin,JwandliHuNonGwhen}, especially providing criteria for the validity of the finitistic dimension conjecture and the Wakamatsu tilting conjecture.

The main aim of this paper is to investigate the homological behavior of modules satisfying the
Auslander condition. More precisely, we first characterize when the class $\mathcal{AC}$ of left $R$-modules satisfying the Auslander condition is resolving. As a result, we establish the connection between the Auslander and Reiten conjecture
mentioned above with the cotorsion pair induced by the class $\mathcal{AC}^{<\infty}$ of left $R$-modules with finite $\mathcal{AC}$-dimension.

To state our results precisely, we first introduce some notation and definitions.
Throughout this paper, $R$ is an associative ring with identity and $R$-Mod is the category of left $R$-modules. Recall that a full subcategory $\mathcal{C}$ of $R$-Mod is \emph{resolving} if  it contains all projective modules and is closed under direct summands, extensions and kernels of surjective homomorphisms. The notion of a resolving subcategory was introduced by Auslander and Bridger \cite{auslander:smt} in the study of totally reflexive modules, which are also called modules of Gorenstein dimension zero or finitely generated Gorenstein projective modules.

Let $M$ be a left $R$-module. We use
$$0\to M \to E^0(M) \to E^1(M) \to \cdots \to E^i(M) \to \cdots$$
to denote a minimal injective coresolution of $M$.
Following \cite{Huang Auscondition}, $M$ is
said to satisfy \emph{the Auslander condition} if the flat dimension of $E^i(M)$ is at most $i$ for any $i\geqslant0$. Denote by $\mathcal{AC}$ the class of left $R$-modules satisfying the Auslander condition.

Recall that a ring $R$ is \emph{Gorenstein} \cite{Iwer} if it is left and right Noetherian and $R$ has finite self-injective dimension on either side. Moreover, a Gorenstein ring is called \emph{Auslander-Gorenstein} \cite{BjAusre} if it satisfies the Auslander condition.
A left and right Noetherian ring $R$ is called \emph{Auslander-regular} \cite{BjAusre} if $R$ satisfies the Auslander condition  and each left $R$-module has finite projective dimension.

Our first main result, characterizing when $\mathcal{AC}$ is a resolving subcategory of $R$-Mod, can be stated as follows.

\begin{thm}\label{thm:1.1}
Let $R$ be a left Noetherian ring. Then the following are equivalent:
\begin{enumerate}
\item $\mathcal{AC}$ is a resolving subcategory of $R$-{\rm Mod}.
\item  $_RR$ satisfies the Auslander condition.
 \item  Each Gorenstein  projective left $R$-module satisfies the Auslander condition.
 \item  Each Ding  projective left  $R$-module satisfies the Auslander condition.
 \item  Each Gorenstein flat left $R$-module satisfies the Auslander condition.
 \end{enumerate}
\end{thm}
The proof of Theorem \ref{thm:1.1} relies heavily on the observation that
each cycle of an exact complex with each term in $\mathcal{AC}$ belongs to $\mathcal{AC}$ for any ring $R$; see Proposition \ref{prop: propac201}. As a consequence of Theorem \ref{thm:1.1}, we prove that an Artinian algebra $R$ is Auslander-Gorenstein if and only if $\mathcal{AC}$ coincides with the class of Gorenstein projective left $R$-modules, which gives a criterion for the validity of the Auslander and Reiten conjecture mentioned above; see Corollary \ref{cor: corarg1}.

Recall that a cotorsion pair $(\mathcal{A},\mathcal{B})$ of left $R$-modules is said to be \emph{tilting-like} if $\mathcal{B}=(T\oplus S)^{\perp_{\infty}}$, where $T$ is a tilting left $R$-module and $S$ is a strongly Gorenstein projective left $R$-module; see \cite[Definition 1.19]{MoraSaFin}. Needless to say, for $S=0$ we recover the tilting cotorsion pair induced by $T$.

Let $\mathcal{AC}^{< \infty}$ be the class of left $R$-modules with finite  $\mathcal{AC}$-dimension, i.e., for each $M$ in $\mathcal{AC}^{< \infty}$, there is a non-negative integer $m$  and an exact sequence $0\ra P_m\ra \cdots \ra P_0\ra M\ra 0$ of left $R$-modules with each $P_i$ in $\mathcal{AC}$. Assume that $R$ is a left Noetherian ring such that $_{R}R$ satisfies the Auslander condition. It follows from Theorem \ref{thm:1.1} that $\mathcal{AC}$ is a resolving subcategory of $R$-Mod. In combination with a result due to Huang \cite[Theorem 3.2]{HuangHdr}, we obtain that $\mathcal{AC}^{< \infty}$ is also resolving; see Corollary \ref{cor:corACfres}.

Our next main result characterizes when $\mathcal{AC}^{< \infty}$ can induce tilting or tilting-like cotorsion pairs.
\begin{thm} \label{thm:1.2} Let $R$ be a  left and right Noetherian  ring satisfying the Auslander condition. Denote by $\mathcal{I}$ the class of injective left $R$-modules and by $\mathcal{GI}$ the class of Gorenstein injective left $R$-modules. Then
 \begin{enumerate}
 \item $R$ is a Gorenstein ring if and only if {\rm(}${\mathcal{AC}^{< \infty}},\mathcal{I}${\rm)} is a tilting-like cotorsion pair.

 \item The following conditions are equivalent:
 \begin{enumerate}
 \item $R$ is an Auslander-regular ring.
  \item {\rm(}${\mathcal{AC}^{<  \infty}},\mathcal{I}${\rm)} is a tilting cotorsion pair.
  \item {\rm(}${\mathcal{AC}^{<  \infty}},\mathcal{GI}${\rm)} is a tilting cotorsion pair.
  \item {\rm(}${\mathcal{AC}^{<  \infty}},\mathcal{GI}${\rm)} is a tilting-like cotorsion pair.
 \end{enumerate}

\item If $R$ is an  Artinian algebra, then $R$ is Gorenstein if and only if {\rm(${\mathcal{AC}^{<  \infty}},(\mathcal{AC}^{<  \infty})^\bot$)} is a tilting-like cotorsion pair.

\item If $R$ is an  Artinian algebra, then $R$ is Auslander-regular if and only if {\rm$({\mathcal{AC}^{<  \infty}},(\mathcal{AC}^{<  \infty})^\bot$)} is a tilting cotorsion pair.
 \end{enumerate}
\end{thm}
Remark that (1) and (3) in Theorem \ref{thm:1.2} together provide criteria for the validity of the Auslander and Reiten conjecture mentioned above.

This paper is structured as follows. Section \ref{pre} contains notation and
definitions for use throughout this paper. In Section \ref{proof}, we prove all results mentioned in Introduction. In the course of our proofs, we prove that each cycle of an exact complex with each term in $\mathcal{AC}$ belongs to $\mathcal{AC}$ for any ring $R$; see Proposition \ref{prop: propac201}.

\section {Preliminaries}\label{pre}
In this section we give some notation and definitions which are used in the paper.

\subsection{Notation} For any left
$R$-module $M$, fd$_{R}M$ (resp. pd$_{R}M$) is
the flat (resp. projective) dimension of $M$.

Let $\mathcal{C}$ be a full subcategory of $R$-Mod. Then we have
$\mathcal{C}^\bot=\{M\in R{\text-}{\rm Mod}\ |\ \Ext_{R}^1(C, M)=0\ {\rm for\ all}\ C\in\mathcal{C}\},$
$^\bot\mathcal{C}=\{M\in R{\text-}{\rm Mod}\ |\ \Ext_{R}^1(M, C)=0\ {\rm for\ all}\ C\in\mathcal{C}\},$
$\mathcal{C}^{\bot_{\infty}}=\{M\in R{\text-}{\rm Mod}\ |\ \Ext_{R}^i(C, M)=0\ {\rm for\ all}\ C\in\mathcal{C}\ {\rm and\ all}\ i\geqslant1 \}.$

For any left $R$-module $M$, we use
$0\ra M\ra E^0(M)\ra \cdots \ra E^i(M)\ra \cdots$
to denote a minimal injective coresolution of $M$.

Let $\xymatrix{ \cdots\ar[r]&X_{n+1}\ar[r]&X_{n}\ar[r]^{\partial_{n}}&X_{n-1}\ar[r]&\cdots}$ be a complex of left $R$-modules. We will denote by $\ker(\partial_n)$ the $n$-th cycle.

\subsection{Gorenstein modules} Following \cite{EJ GP,H}, a left $R$-module $G$ is called \emph{Gorenstein projective} if there is an  exact sequence $$\mathbf{P}:\cdots\ra P_1\ra P_0\ra P^0\ra P^1\ra \cdots$$ of projective left $R$-modules such that $G$ $\cong$ $\ker(P^{0}\ra P^1)$ and $\Hom_R(\mathbf{P},Q)$ is exact for every projective left $R$-module $Q$. In this case, the complex $\mathbf{P}$ is also called a totally acyclic complex of projective left $R$-modules. Moreover, a left $R$-module $M$ is called \emph{strongly Gorenstein projective} \cite{DB} if there is a totally acyclic complex  $$\cdots\ra P\xrightarrow{f} P\xrightarrow{f} P\xrightarrow{f} P\xrightarrow{f}\cdots$$ of projective left $R$-modules such that $M$ $\cong$ $\ker(f)$.

A left $R$-module $M$ is called \emph{Ding projective} \cite{DandLiMStronglyG,Gillespie Dp} if there is an exact sequence  $$\mathbf{P}:\cdots\ra P_1\ra P_0\ra P^0\ra P^1\ra \cdots$$ of projective left $R$-modules with $M\cong \textrm{ker}(P^{0}\ra P^{1})$  such that $\Hom_{R}(\mathbf{P},F)$ is exact for every flat left $R$-module $F$.

A left $R$-module $M$ is called \emph{Gorenstein flat} \cite{EJ} if there is an exact sequence  $$\mathbf{F}:\cdots\ra F_1\ra F_0\ra F^0\ra F^1\ra \cdots$$ of flat left $R$-modules with $M\cong \textrm{ker}(F^{0}\ra F^{1})$  such that $E\otimes_{R}\mathbf{F}$ is exact for every injective right $R$-module $E$.

A left $R$-module $M$ is called \emph{Gorenstein injective} \cite{EJ GP} if there exists an exact sequence  $$\mathbf{I}:\cdots\ra I_1\ra I_0\ra I^0\ra I^1\ra \cdots$$ of injective left $R$-modules with $M\cong \textrm{ker}(I^0\ra I^1)$ such that $\Hom_{R}(E,\mathbf{I})$ is exact for every injective left $R$-module $E$.

In the following, we use $\mathcal{GP}$, $\mathcal{G}p$, $\mathcal{GF}$ and $\mathcal{GI}$ to denote the classes of Gorenstein projective, finitely generated Gorenstein projective, Gorenstein flat and Gorenstein injective left $R$-modules, respectively.

\subsection{Cotorsion pairs}
  Let $\mathcal{A}$ and $\mathcal{B}$ be classes in $R$-Mod.
  A pair $(\mathcal{A},\mathcal{B})$ is called a \emph{cotorsion pair} \cite{GT,SL} if $\mathcal{A}$ $=$ $^\bot\mathcal{B}$ and $\mathcal{B}$ $=$ $\mathcal{A}^\bot$. The class $\mathcal{A}\cap \mathcal{B}$ is called the \emph{kernel} of the cotorsion pair $(\mathcal{A},\mathcal{B})$.

   A cotorsion pair ($\mathcal{A}$, $\mathcal{B}$) is said to be
\emph{hereditary} \cite{GT} if $\Ext_{R}^{i}(A,B)$ $=$ $0$ for  all $i$ $\geqslant$ $1$, $A$ $\in$ $\mathcal{A}$ and $B$ $\in$ $\mathcal{B}$. This definition is equivalent to the condition that if whenever $0\ra A_{3}\ra A_{2} \ra A_{1}\ra 0$ is exact with $A_{2}$, $A_{1}$ $\in$ $\mathcal{A}$, then $A_{3}$ is also in $\mathcal{A}$, or equivalently, if whenever $0\ra B_{1}\ra B_{2} \ra B_{3}\ra 0$ is exact with $B_{1}$, $B_{2}$ $\in$ $\mathcal{B}$, then $B_{3}$ is also in $\mathcal{B}$.

 A cotorsion pair $(\mathcal{A},\mathcal{B})$ is \emph{complete} \cite{GT} if, for each left $R$-module $M$, there is an exact sequence $0\ra M\ra B\ra L\ra 0$ with $B$ $\in$ $\mathcal{B}$ and $L$ $\in$ $\mathcal{A}$. This definition is equivalent to the condition that, for each left $R$-module $M$, there is an exact sequence $0\ra D\ra C\ra M\ra 0$ with $C$ $\in$ $\mathcal{A}$ and $D$ $\in$ $\mathcal{B}$.

A cotorsion
pair $(\mathcal{A},\mathcal{B})$  is \emph{generated by a set} \cite{GT} provided that there is a set $\mathcal{S}$ of left $R$-modules such that $\mathcal{S}^\bot=\mathcal{B}$ $($i.e., $(\mathcal{A},\mathcal{B})$ $=$ $(^\bot(\mathcal{S}^\bot), \mathcal{S}^\bot)$$)$. A cotorsion pair generated by a set $\mathcal{S}$ is also called a cotorsion pair\emph{ cogenerated} by $\mathcal{S}$ in some literature. Here, we use the terminology from \cite{GT}.
If ($\mathcal{A},\mathcal{B}$) is a hereditary cotorsion pair generated by a set $\mathcal{S}$, then there is a single module $M:=\bigoplus_{A\in\mathcal{S}}A$ such that $M^{\bot_\infty}=\mathcal{B}$.

A left $R$-module $T$ is \emph{tilting} \cite{LAHFUCol,Colby1995} if $\pd_RT$ $<$ $\infty$, $T^{(\lambda)}$ $\in$ $T^{\bot_\infty}$ for every cardinal $\lambda$ and there is an exact sequence $0\ra R\ra T_{0}\ra \cdots\ra T_{r}\ra 0$ with $T_{i}$ $\in$ $\Add T$ for all $0$ $\leqslant$ $i$ $\leqslant$ $r$, where $r$ is the projective dimension of $T$ and Add$T$ is all direct summands of direct sums of copies of $T$.

Following \cite{GT}, a cotorsion pair $(\mathcal{A},\mathcal{B})$ of left $R$-modules is called a \emph{tilting cotorsion pair} if $\mathcal{B}$ $=$ $T^{\bot_\infty}$, where $T$ is a tilting left $R$-module. Recall from \cite[Definition 1.19]{MoraSaFin} that a cotorsion pair $(\mathcal{A},\mathcal{B})$ of left $R$-modules is said to be \emph{tilting-like} if $\mathcal{B}=(T\oplus S)^{\perp_{\infty}}$, where $T$ is a tilting left $R$-module and $S$ is a strongly Gorenstein projective left $R$-module. It is clear that all tilting cotorsion pairs are tilting-like cotorsion pairs. But the converse is not true in general. For example, the module $M :$= $\mathbb{Z}/2\mathbb{Z}$ over the
quasi-Frobenius ring $\mathbb{Z}/4\mathbb{Z}$ is clearly a non-projective strongly Gorenstein
projective module, whence the hereditary cotorsion pair generated by $M$ is
tilting-like, but it cannot be tilting since $M$ does not have finite projective
dimension.

\section{Proofs of the statements}\label{proof}
In this section, we prove all results mentioned in Introduction. We keep the notation introduced in the previous sections.

\subsection{Proof of Theorem \ref{thm:1.1}}
This subsection is devoted to the proof of Theorem \ref{thm:1.1}. For this purpose, we need some
technical results.

Let $0\ra M\ra E^0(M)\ra \cdots \ra E^i(M)\ra\cdots$ be a minimal injective coresolution of a left $R$-module $M$. If there is an exact sequence $0\ra M\ra E^0\ra\cdots\ra E^i\ra \cdots $ of left $R$-modules with each $E^i$ injective, then one can check that $E^i(M)$ is a direct summand of $E^i$  for all $i$ $\geqslant$ $0$. We will use this fact in the following.

\begin{lem}\label{lem:lem1E}  Let $R$ be a ring and $0\ra M\ra P\ra N\ra 0$ an exact sequence of left $R$-modules. Then the following are true:
 \begin{enumerate}
\item $E^i(P)$ is a direct summand of $E^i(M)\oplus E^i(N)$ for all $i$ $\geqslant$ $0$.
\item $E^i(M)$ is a direct summand of $E^{i-1}(E^0(M) \oplus N)\oplus E^{i}(P)$ for all $i$ $\geqslant$ $1$.
\end{enumerate}
\end{lem}
\begin{proof} (1) By the Horseshoe Lemma, we have the following commutative diagram with exact rows and columns
$$\xymatrix{&0\ar[d]&0\ar[d]&0\ar[d]&\\
0\ar[r]&M\ar[r]\ar[d]&P\ar[r]\ar[d]&N\ar[d]\ar[r]&0\\
0\ar[r]&E^0(M)\ar[d]\ar[r]&E^0(M)\oplus E^0(N)\ar[r]\ar[d]&E^0(N)\ar[d]\ar[r]&0\\
0\ar[r]&K\ar[r]\ar[d]&T\ar[r]\ar[d]&S\ar[d]\ar[r]&0\\
 &0&0&0&}$$

Repeating this procedure, we obtain an exact sequence $$0\ra P\ra E^0(M) \oplus E^0(N) \ra \cdots \ra E^i(M)\oplus E^i(N)\ra \cdots $$ of left $R$-modules. Using the exact sequence, one can check that $E^i(P)$ is a direct summand of $E^i(M)\oplus E^i(N)$ for all $i$ $\geqslant$ $0$.

(2) Note that there is an exact sequence $0\ra M\ra E^0(M)\ra K\ra 0$ of left $R$-modules. Hence we have the following commutative diagram with exact rows and columns
$$\xymatrix{&0\ar[d]&0\ar[d]&&\\
0\ar[r]&M\ar[d]\ar[r]&E^0(M)\ar[d]\ar[r]&K\ar@{=}[d]\ar[r]&0\\
0\ar[r]&P\ar[r]\ar[d]&D\ar[r]\ar[d]&K\ar[r]&0\\
&N\ar@{=}[r]\ar[d]&N\ar[d]&&\\
 &0&0&&}$$
Since $D$ $\cong$ $E^0(M) \oplus N$, there is an exact sequence $0\ra P\ra E^0(M) \oplus N\ra K\ra 0$ of left $R$-modules. Applying \cite[Theorem 3.4]{Huang Proper}, one can obtain an exact sequence
 \begin{equation}\label{3.1}0\ra E^0(P)\ra E^0(E^0(M)\oplus N)\oplus E^1(P)\ra C\ra 0
 \end{equation}
 of left $R$-modules, where $C$ comes from the  exact sequence
\begin{equation}\label{3.2}0\ra K\ra C\ra E^1(E^0(M) \oplus N)\oplus E^2(P)\ra \cdots \ra E^i(E^0(M) \oplus N)\oplus E^{i+1}(P)\ra \cdots
\end{equation}
of left $R$-modules.

 By the sequence \eqref{3.1}, $C$ is a direct summand of $E^0(E^0(M)\oplus N)\oplus E^1(P)$, and so $C$ is injective.  Then $E^0(K)$ is a direct summand of $C$. It follows that $E^0(K)$ is a direct summand of $E^0(E^0(M)\oplus N)\oplus E^1(P)$. Since $C$ is injective, it follows from the sequence \eqref{3.2} that $E^{i}(K)$ is a direct summand of $E^{i}(E^0(M) \oplus N)\oplus E^{i+1}(P)$ for all $i$ $\geqslant$ $1$. Thus
$E^{i-1}(K)$ is a direct summand of $E^{i-1}(E^0(M) \oplus N)\oplus E^{i}(P)$ for all $i$ $\geqslant$ $1$. Since $E^{i}(M)$ $\cong$ $E^{i-1}(K)$, $E^i(M)$ is a direct summand of $E^{i-1}(E^0(M) \oplus N)\oplus E^{i}(P)$ for all $i$ $\geqslant$ $1$.
\end{proof}

The following is a key result to prove Theorem \ref{thm:1.1}.

\begin{prop} \label{prop: propac201} Let $R$ be a ring. Then each cycle of an exact complex with each term in $\mathcal{AC}$ belongs to $\mathcal{AC}$.
\end{prop}
\begin{proof} Suppose that there is an exact complex $$\xymatrix{\mathbf{X}:\cdots\ar[r]&P_{n}\ar[r]^{\partial_{n}}&P_{n-1}\ar[r]^{\partial_{n-1}}&P_{n-2}\ar[r]&\cdots}$$  of left $R$-modules with each $P_{i}$ in $\mathcal{AC}$. For any cycle $M$ of the exact complex $\mathbf{X}$, we need to prove that fd$_R E^i(M)$ $\leqslant$ $i$ for all $i$ $\geqslant$ $0$. We shall prove it by induction on $i$.

Let $M$ $\cong$ $\ker(\partial_{n})$ be the $n$-th cycle. Then there is an exact sequence $0\ra M\ra P\ra N\ra 0$ with $P$ $\cong$ $P_{n}$ and $N$ $\cong$ $\ker(\partial_{n-1})$.

For $i=0$,  $E^0(M)$ is a direct summand of $E^0(P)$ by \cite[Proposition 18.12]{AF}. Since $P$ is in ${\mathcal{AC}}$, $E^0(P)$ is flat. Thus $E^0(M)$ is flat, as desired.

For $i$ $\geqslant$ $1$,  assume that  fd$_R E^{i-1}(M) \leqslant i-1$ for any cycle $M$ of the complex $\mathbf{X}$. We will prove that fd$_R E^{i}(M)$ $\leqslant$ $i$. Note that $E^0(E^0(M)\oplus N)$ $\cong$ $E^0(M)\oplus E^0(N)$. One can check that $E^{i-1}(E^0(M) \oplus N)$ $\cong$ $E^{i-1}(N)$ for all $i\geqslant2$. Since $M$ and $N$ are cycles of the exact complex $\mathbf{X}$, fd$_R E^{i-1}(E^0(M) \oplus N)$ $\leqslant$ $i-1$ by the inductive hypothesis. By assumption, $P$ is in ${\mathcal{AC}}$. It follows that  fd$_R E^{i}(P)$ $\leqslant$ $i$. Then fd$_R (E^{i-1}(E^0(M) \oplus N)\oplus E^{i}(P))$ $\leqslant$ $i$. By Lemma \ref{lem:lem1E}, each $E^i(M)$ is a direct summand of $E^{i-1}(E^0(M) \oplus N)\oplus E^{i}(P)$. It follows that fd$_R E^i(M)$ $\leqslant$ $i$. This completes the proof.
\end{proof}

Now we give the proof of Theorem \ref{thm:1.1}.

{\bf Proof of Theorem \ref{thm:1.1}.} (1) $\Rightarrow$ (2), (3) $\Rightarrow$ (2), (4) $\Rightarrow$ (2) and (5) $\Rightarrow$ (2) are clear.

(2) $\Rightarrow$ (1). Since $R$ is left Noetherian and $R$ is in $\mathcal{AC}$ as a left $R$-module, all projective left $R$-modules are in $\mathcal{AC}$ by (2) and \cite[Theorem 4.9]{Huang Auscondition}. Let $0\ra C_1\ra C_2\ra C_3 \ra 0$  be an exact sequence of left $R$-modules. Then there exists an exact sequence $$0\ra C_1\ra C_2\ra C_3\oplus C_3 \ra C_3\oplus C_3\ra \cdots$$ of left $R$-modules.
We choose an exact sequence $$\cdots \ra P_n\ra \cdots\ra P_0\ra C_1\ra 0 $$ of left $R$-modules with each $P_i$ projective. Thus we obtain an exact complex $$\mathbf{X}:\cdots \ra P_n\ra \cdots\ra P_0\ra C_2\ra C_3\oplus C_3\ra  C_3\oplus C_3 \ra \cdots $$ with  $C_1$ a cycle of the exact complex $\mathbf{X}$. If $C_2$ and $C_3$ are in $\mathcal{AC}$, then each term of the exact complex $\mathbf{X}$ belongs to $\mathcal{AC}$ by noting that $\mathcal{AC}$ is closed under finite direct sums and all projective left $R$-modules are in $\mathcal{AC}$. Hence $C_1$ is in $\mathcal{AC}$ by Proposition \ref{prop: propac201}. By Lemma \ref{lem:lem1E}, $\mathcal{AC}$ is closed under extensions. One can check that $\mathcal{AC}$ is closed under direct sums. Then $\mathcal{AC}$ is closed under direct summands by \cite[Proposition 1.4]{H}. So $\mathcal{AC}$ is resolving, as desired.

(1) $\Rightarrow$ (3). Note that each Gorenstein projective left $R$-module is a cycle of an exact complex of projective left $R$-modules. Thus (3) follows from (1) and Proposition \ref{prop: propac201}.

(1) $\Rightarrow$ (4). The proof is similar to that (1) $\Rightarrow$ (3).

(2) $\Rightarrow$ (5). Note that all flat left $R$-modules are in $\mathcal{AC}$ by (2) and \cite[Theorem 4.9]{Huang Auscondition}.  Actually, each Gorenstein flat left $R$-module is a cycle of an exact complex of flat left $R$-modules. Thus (5) follows from Proposition \ref{prop: propac201}.
\hfill$\Box$

\vspace{2mm}

Let $n$ be a non-negative integer. Recall that a left $R$-module $G$ has \emph{finite Gorenstein flat dimension at most $n$} if there is an exact sequence $0\ra G_{n}\ra \cdots \ra G_{0}\ra G\ra 0$ of left $R$-modules with all $G_{i}$ Gorenstein flat (see \cite{H}). If there is no need to refer the number $n$, we shall simply say that $G$ has finite Gorenstein flat dimension. In the following, we denote by $\mathcal{A}c$ the class of finitely generated left $R$-modules satisfying the Auslander condition.
\begin{cor}\label{cor:cor1gfa} Let $R$ be Auslander-Gorenstein.  Then the following are true:
\begin{enumerate}
\item $\mathcal{AC}$ $=$ $\mathcal{GF}$.
\item $\mathcal{A}c$ $=$ $\mathcal{G}p$.
\end{enumerate}
\end{cor}
\begin{proof} (1) By Theorem \ref{thm:1.1}, $\mathcal{GF}$ $\subseteq$ $\mathcal{AC}$. Next we check $\mathcal{GF}$ $\supseteq$ $\mathcal{AC}$. Suppose that $M$ is  a left $R$-module in $\mathcal{AC}$. Since $R$ is Gorenstein, the injective dimension of $R$ as a left $R$-module is finite. Let $n$ be the injective dimension. Since $M$ is in $\mathcal{AC}$, there is an exact sequence $0\ra M\ra E^0(M)\ra \cdots \ra E^{n-1}(M)\ra L\ra 0$ of left $R$-modules such that fd$_R E^i(M)$ $\leqslant$ $i$ for all $i$ $\geqslant$ $0$. By \cite[Theorem 12.3.1]{EJ}, the Gorenstein flat dimension of $L$ is at most $n$. It follows from \cite[Theorem 3.14]{H} that $\Tor_{n+k}^R(E,L)$ $=$ $0$ for each injective right $R$-module $E$ and all $k$ $\geqslant$ $1$. Using dimension shifting, one can check that $\Tor_{k}^R(E,M)$ $=$ $0$ for each injective right $R$-module $E$ and all $k$ $\geqslant$ $1$. Note that the Gorenstein flat dimension of $M$ is finite by \cite[Theorem 12.3.1]{EJ}. Hence $M$ is Gorenstein flat by \cite[Theorem 3.14]{H}.

(2) By \cite[Corollary 10.3.11]{EJ},  a finitely generated left $R$-module is Gorenstein projective if and only if it is Gorenstein flat. By (1), we obtain the desired result.
\end{proof}

\begin{rem} \label{rem: remAgp1} {It follows from \cite{Bass1963} that a commutative Gorenstein ring is Auslander-Gorenstein. Thus, $\mathcal{AC}$ $=$ $\mathcal{GF}$ and  $\mathcal{A}c$ $=$ $\mathcal{G}p$ over any commutative Gorenstein ring $R$.}
\end{rem}

Let $\cdots \ra P_i\xrightarrow{f_i}\cdots\ra P_0\xrightarrow{f_0} M\ra0$ be a projective resolution of a left $R$-module $M$. The module $\im{(f_{i})}$ is called an \emph{$i$-th syzygy} of $M$, denoted by $\Omega^{i}M$ ($\Omega^{0}M$ $=$ $M$).

Let $n$ be a non-negative integer. The \emph{Gorenstein projective dimension}, $\Gpd_{R}G$, of a left $R$-module $G$ is defined by declaring that $\Gpd_{R}G$ $\leqslant$ $n$ if, and only if there is an exact sequence $0\ra G_{n}\ra \cdots \ra G_{0}\ra G\ra 0$ of left $R$-modules with all $G_{i}$ Gorenstein projective (see \cite{H}). We will say that $G$ has finite Gorenstein projective dimension if there is no need to refer the number $n$.

Note that $\Gpd_{R}M$ $\leqslant n$ if and only if $\Omega^n(M)$ is Gorenstein projective (see \cite[Proposition 2.7]{H}). Moreover, if the projective dimension of $M$ is finite, then Gpd$_R M$ $=$ pd$_R M$ (see \cite[Proposition 2.27]{H}).

\begin{cor} \label{cor: corarg1} Let $R$ be an  Artinian algebra. Then the following are equivalent:
\begin{enumerate}
\item $R$ is Auslander-Gorenstein.
\item $\mathcal{GP}$ $=$ $\mathcal{AC}$.
\item $\mathcal{G}p$ $=$ $\mathcal{A}c$.
\end{enumerate}
\end{cor}
\begin{proof}(1) $\Rightarrow$ (2) follows from Corollary \ref{cor:cor1gfa} since $\mathcal{GP}$ $=$ $\mathcal{GF}$ over an  Artinian algebra.

(2) $\Rightarrow$ (3) is clear.

(3) $\Rightarrow$ (1). Note that $R$ is in $\mathcal{A}c$ as a left $R$-module by (3). Let $$0\ra R\ra E^0(R)\ra \cdots \ra E^i(R)\ra E^{i+1}(R)\ra \cdots$$ be a minimal injective coresolution of $R$.  Then there is a non-negative integer $n$ such that pd$_R E^i(R)$ $\leqslant$ $n$ for all $i$ $\geqslant$ $0$ by noting that an  Artinian algebra has only finite number of non-isomorphic indecomposable finitely generated injective left $R$-modules. Using the minimal injective coresolution of $R$, we can obtain another exact complex $$\cdots\ra R\oplus R \ra R\oplus R \ra E^0(R)\ra \cdots \ra E^i(R)\ra E^{i+1}(R)\ra \cdots.$$ Using the Horseshoe Lemma, we obtain an exact complex $$\mathbf{X}:\cdots\ra\Omega^n(R\oplus R)\ra\Omega^n(R\oplus R)\ra \Omega^n(E^0(R))\ra \cdots \ra \Omega^n(E^i(R)) \ra \Omega^n(E^{i+1}(R))\ra\cdots $$ of left $R$-modules with
\begin{equation}\label{3.3}\ker(\Omega^n(E^i(R)) \ra \Omega^n(E^{i+1}(R))) \cong \Omega^n(\ker(E^i(R) \ra E^{i+1}(R)))
\end{equation}
for all $i$ $\geqslant$ $0$. By the proof above,  $\Omega^n(E^i(R))$ is projective for all $i\geqslant0$. Note that all syzygies in the exact complex $\mathbf{X}$ can be taken to be finitely generated. Thus each term in the exact complex  $\mathbf{X}$ is in $\mathcal{A}c$ by (3).  Applying Proposition \ref{prop: propac201}, ker$(\Omega^n(E^i(R)) \ra \Omega^n(E^{i+1}(R)))$ is in $\mathcal{A}c$ for all $i$ $\geqslant$ $0$. It follows from (3) that  ker$(\Omega^n(E^i(R)) \ra \Omega^n(E^{i+1}(R)))$ is Gorenstein projective for all $i$ $\geqslant$ $0$. Using the isomorphism \eqref{3.3}, one can obtain that ker$(E^i(R)\ra E^{i+1}(R))$ has finite Gorenstein projective dimension at most $n$ for all $i$ $\geqslant$ $0$. Let $L_i$ $=$ ker$(E^i(R)\ra E^{i+1}(R))$ for all $i$ $\geqslant$ $0$. Then Gpd$_R L_{n+1}$ $\leqslant$ $n$, and so  Ext$_R^{n+1}(L_{n+1}, R)$ $=$ $0$ by \cite[Theorem 2.20]{H}.  Using dimension shifting, one can check that Ext$_R^{1}(L_{n+1}, L_n)$ $=$ $0$. It follows that the exact sequence $0\ra L_n\ra E^{n}(R)\ra L_{n+1} \ra 0$ splits. Then $L_{n}$ is injective, and therefore $R$ has finite injective dimension at most $n$ as a left $R$-module.  So $R$ is Auslander-Gorenstein by \cite[Corollary 5.5(b)]{ARk-G}.
\end{proof}

\begin{cor}\label{cor:corACfres} If $R$ is a left Noetherian ring such that $_{R}R$ satisfies the Auslander condition, then $\mathcal{AC}^{<\infty}$ is a resolving subcategory of {\rm $R$-Mod}.
\end{cor}
\begin{proof}
 By Theorem \ref{thm:1.1}, $\mathcal{AC}$ is a resolving subcategory of $R$-Mod. To prove that $\mathcal{AC}^{<\infty}$ is a resolving subcategory of $R$-Mod,  we only need to check that $\mathcal{AC}^{<\infty}$ is closed under direct summands by \cite[Theorem 3.2]{HuangHdr}. Note that $\mathcal{AC}^{<\infty}$ $=$ $\bigcup_{s\geqslant0} \mathcal{AC}^{\leqslant s}$, where $\mathcal{AC}^{\leqslant s}$ is the class of left $R$-modules with finite $\mathcal{AC}$-dimension at most $s$, i.e., for each $M$ in $\mathcal{AC}^{\leqslant s}$, there is an exact sequence $0\ra P_s\ra \cdots \ra P_0\ra M\ra 0$ of left $R$-modules with each $P_i$ in $\mathcal{AC}$. Applying \cite[Theorem 3.2]{HuangHdr}, we obtain that $\mathcal{AC}^{\leqslant s}$ is closed under extensions and kernels of surjective homomorphisms for all $s\geqslant0$. Since $\mathcal{AC}^{\leqslant s}$ is closed under direct sums, it follows from \cite[Proposition 1.4]{H} that $\mathcal{AC}^{\leqslant s}$ is closed under direct summands. This completes the proof.
\end{proof}
\subsection{Proof of Theorem \ref{thm:1.2}}
In the following, we will give the proof of Theorem \ref{thm:1.2}.  Recall that the \emph{left (little) finitistic dimension} of a ring $R$ is the supremum of projective dimensions of (finitely generated) left $R$-modules of finite projective dimension. It is conjectured that the left little finitistic dimension of each  Artinian algebra is finite; see \cite{Bass} and \cite[Conjectures]{ARS}. This is the finitistic dimension conjecture.

\begin{lem} \label{prop: propfd1} Let $R$ be a  ring and {\rm($\mathcal{A},\mathcal{B} $)} a tilting-like cotorsion pair. If all left $R$-modules with finite projective dimension are in $\mathcal{A}$, then the left finitistic dimension of $R$ is finite. Moreover, if $R$ is left and right Noetherian satisfying the Auslander condition, then injective dimension of $R$ as a left $R$-module is finite.

\end{lem}
\begin{proof} Let ($\mathcal{A},\mathcal{B} $) be a tilting-like cotorsion pair. Then there is a tilting left $R$-module $T$ and a strongly Gorenstein projective left $R$-module $S$ such that $\mathcal{B}$ $=$ $(T\oplus S)^{\bot_\infty}$. Let $n$ be the projective dimension of $T$.  By \cite[Theorem 1.1]{JwandliHuNonGwhen}, each left $R$-module in  $\mathcal{A}$ has finite Gorenstein projective dimension at most $n$.  Let $M$ be any left $R$-module with finite projective dimension. Then $M$ is in ${\mathcal{A}}$ by assumption. Hence the Gorenstein projective dimension of $M$ is at most $n$, and so the projective dimension of $M$ is at most $n$. It follows that the left finitistic dimension of $R$ is finite at most $n$, which implies that the left little finitistic dimension of $R$ is finite. If $R$ is left and right Noetherian satisfying the Auslander condition, then the injective dimension of $R$ as a left $R$-module is finite by \cite[Corollary 5.3]{Huang I cot}.
\end{proof}

Following \cite{H S app C P}, a full subcategory $\mathcal{C}$ of $R$-Mod has the \emph{two out of three property} if for any exact sequence $0\rightarrow A\rightarrow B\rightarrow C \rightarrow 0$  of left $R$-modules  with two terms in $\mathcal{C}$, then the third one is also in $\mathcal{C}$.

\begin{lem} \label{prop: propGIr1} Let $R$ be a left and right  Noetherian ring satisfying the Auslander condition. Then the following are equivalent:
 \begin{enumerate}
 \item $R$ is Gorenstein.
 \item {\rm(${^\bot\mathcal{GI}},\mathcal{GI}$)} is a tilting-like cotorsion pair.
 \item There is a tilting-like cotorsion pair {\rm($\mathcal{A},\mathcal{B}$)} such that all left $R$-modules with finite projective dimension are in $\mathcal{A}$ and all injective left $R$-modules are in the kernel of the cotorsion pair.
 \end{enumerate}
\end{lem}
\begin{proof} (1) $\Rightarrow$ (2). By \cite[Theorem 5.6]{JJarxiv}, (${^\bot\mathcal{GI}},\mathcal{GI}$) is a complete hereditary cotorsion pair generated by a set of left $R$-modules. Then there is a left $R$-module $M$ in ${^\bot\mathcal{GI}}$ such that $M^\bot$ $=$ $\mathcal{GI}$. Note that $M^{\bot_{\infty}}$ $=$ $M^\bot$ since (${^\bot\mathcal{GI}},\mathcal{GI}$) is a hereditary cotorsion pair.  Then  $M^{\bot_\infty}$ $=$ $\mathcal{GI}$. One can check that the kernel of (${^\bot\mathcal{GI}},\mathcal{GI}$) is the class $\mathcal{I}$ of injective left $R$-modules. Since $R$ is Gorenstein, each injective left $R$-module has finite projective dimension at most $n$ by [15, Theorem 9.1.10], where $n$ is the injective dimension of the left $R$-module $R$. Note that $\mathcal{I}$ is closed under direct sums since $R$ is a left Noetherian
ring. By \cite[Theorem 12.3.1]{EJ}, each left $R$-module in ${^\bot\mathcal{GI}}$ has finite Gorenstein projective dimension at most $n$. Thus (${^\bot\mathcal{GI}},\mathcal{GI}$) is a tilting-like cotorsion pair by \cite[Theorem 1.1]{JwandliHuNonGwhen}.

(2) $\Rightarrow$ (3). It is clear that all projective left $R$-modules are in ${^\bot\mathcal{GI}}$. By \cite[Lemma 5.4]{JJarxiv},  ${^\bot\mathcal{GI}}$ has two out of three property. Thus all left $R$-modules with finite projective dimension are in ${^\bot\mathcal{GI}}$. By assumption, (${^\bot\mathcal{GI}},\mathcal{GI}$) is a tilting-like cotorsion pair. Note that the kernel of (${^\bot\mathcal{GI}},\mathcal{GI}$) is the class of injective left $R$-modules. This yields that (${^\bot\mathcal{GI}},\mathcal{GI}$) is a desired cotorsion pair satisfying the conditions in (3).

(3) $\Rightarrow$ (1). By Lemma \ref{prop: propfd1}, $R$ has finite injective dimension as a left  $R$-module.  Since ($\mathcal{A},\mathcal{B}$) is a tilting-like cotorsion pair, each left $R$-module in the kernel has finite projective dimension by \cite[Theorem 1.1]{JwandliHuNonGwhen}. Note that the left $R$-module Hom$_\mathbb{Z}(R,\mathbb{Q}/\mathbb{Z})$ is injective, where $\mathbb{Q}$
is the additive group of rational numbers and $\mathbb{Z}$ is the additive group of integers. Then  Hom$_\mathbb{Z}(R,\mathbb{Q}/\mathbb{Z})$ is in the kernel of the cotorsion pair ($\mathcal{A},\mathcal{B} $) by assumption. Hence Hom$_\mathbb{Z}(R,\mathbb{Q}/\mathbb{Z})$ has finite projective dimension. It follows that $R$ has finite injective dimension as a right $R$-module. Thus $R$ is Gorenstein.
\end{proof}

We are now ready to prove Theorem \ref{thm:1.2}.

{\bf Proof of Theorem \ref{thm:1.2}.} (1)
Assume that $R$ is Gorenstein. Let  $\mathcal{GF}{^{< \infty}}$ be the class of left $R$-modules with finite Gorenstein flat dimension. Then  $\mathcal{AC}^{< \infty}$  $=$ $\mathcal{GF}^{<  \infty}$ by Corollary \ref{cor:cor1gfa}.  By \cite[Theorem 12.3.1]{EJ}, $\mathcal{GF}^{<  \infty}$ $=$ $R$-Mod, and so  $\mathcal{AC}^{< \infty}$ $=$ $R$-Mod. Let $n$ be the injective dimension of the left $R$-module $R$ and  $\mathcal{GP}{^{\leqslant n}}$ the class of left $R$-modules with finite Gorenstein projective dimension at most $n$.
 Then $R$-Mod $=$ $\mathcal{GP}^{\leqslant n}$ by \cite[Theorem 12.3.1]{EJ}. It is well known that  ($R$-Mod, $\mathcal{I}$) is a complete  hereditary cotorsion pair and there is a left $R$-module $M$ such that $M^{\bot_\infty}$ $=$ $\mathcal{I}$. By \cite[Proposition 4.8 and Theorem 1.1]{JwandliHuNonGwhen}, (${\mathcal{AC}^{< \infty}},\mathcal{I}$) is a tilting-like cotorsion pair.

Conversely, assume that (${\mathcal{AC}^{<  \infty}},\mathcal{I}$) is a tilting-like cotorsion pair. Note that all left $R$-modules with finite projective dimension are in ${\mathcal{AC}^{<  \infty}}$ and the kernel of (${\mathcal{AC}^{<  \infty}},\mathcal{I}$) is $\mathcal{I}$. Then $R$ is Gorenstein by Lemma \ref{prop: propGIr1}.

(2) (a) $\Rightarrow$ (b). Note that (${\mathcal{AC}^{<  \infty}},\mathcal{I}$) is a tilting-like cotorsion pair by (1).  Since $R$ is Auslander-regular, one can check that there is a non-negative integer $n$ such that each left $R$-module has finite projective dimension at most $n$. So (${\mathcal{AC}^{<  \infty}},\mathcal{I}$) is a tilting cotorsion pair by \cite[Corollary 13.20]{GT}.

(b) $\Rightarrow$ (c). To prove (c), it suffices to show $\mathcal{GI}$ $=$ $\mathcal{I}$. Since (${\mathcal{AC}^{<  \infty}},\mathcal{I}$) is a tilting cotorsion pair, it follows from \cite[Lemma 13.10]{GT} that each left $R$-module has finite projective dimension. By \cite[Corollary 2.9]{GDWGor}, $\mathcal{GI}$ $=$ $\mathcal{I}$, as desired.

(c) $\Rightarrow$ (d) is clear.

(d) $\Rightarrow$ (a). By assumption, (${\mathcal{AC}^{<  \infty}},\mathcal{GI}$) is a tilting-like cotorsion pair. Applying  Lemma \ref{prop: propGIr1}, $R$ is Gorenstein.  Then ${\mathcal{AC}^{<  \infty}}$ $=$ $R$-Mod by (1). Thus each left $R$-module has finite projective dimension by  \cite[Remark 11.2.3]{EJ}, and therefore $R$ is Auslander-regular.

(3) The ``only if" part follows from  (1). For the ``if" part, by Theorem \ref{thm:1.1}, all left $R$-modules with finite projective dimension are in ${\mathcal{AC}^{<  \infty}}$. By Lemma \ref{prop: propfd1}, $R$ has finite injective dimension as a left $R$-module. So $R$ is Gorenstein by \cite[Corollary 5.5(b)]{ARk-G}.

(4) The ``only if" part follows from  (2). For the ``if" part, assume that (${\mathcal{AC}^{<  \infty}},(\mathcal{AC}^{< \infty})^\bot$) is a tilting cotorsion pair, it is also a tilting-like cotorsion pair. Since $R$ is Gorenstein by (3), it follows from \cite[Theorem 12.3.1]{EJ} that each left $R$-module has finite Gorenstein projective dimension. Note that each left $R$-module in ${\mathcal{AC}^{<  \infty}}$ has finite projective dimension by \cite[Lemma 13.10]{GT} and all Gorenstein projective left $R$-modules are in $\mathcal{AC}$ by Theorem \ref{thm:1.1}. It follows that each Gorenstein projective left $R$-module has finite projective dimension, and so each Gorenstein projective left $R$-module is projective. Thus each left $R$-module has finite projective dimension. So $R$ is Auslander-regular.
\hfill$\Box$

\bigskip \centerline {\bf ACKNOWLEDGEMENTS}
\bigskip
This research was partially supported by the Natural Science Foundation of the Jiangsu Higher Education Institutions of China (19KJB110012, 21KJB110003), NSFC (12201264, 12171206),  Jinling Institute of Technology of China (jit-b-201638, jit-fhxm-201707), the Natural Science Foundation of Jiangsu Province (BK20211358) and Jiangsu 333
Project. The authors thank Professor Zhaoyong
Huang from Nanjing University for helpful discussions on parts of this paper. Also, the authors appreciate the referee for reading the paper carefully and for
many suggestions on mathematics and English expressions.

\end{document}